\documentclass{article}

\usepackage[utf8]{inputenc}
\usepackage{epsfig}
\usepackage{amssymb}
\usepackage{amscd}
\usepackage[matrix,arrow]{xy}
\usepackage{graphicx}
\usepackage{amsmath, amsthm}
\usepackage{amsfonts, bbm, mathrsfs}
\usepackage{color, enumerate}
\usepackage[usenames,dvipsnames,svgnames,table]{xcolor, pstricks}
\usepackage{tikz, multicol}
\usepackage{tikz-cd}
\usepackage{float}
\usepackage{mathtools}
\usepackage{enumitem}

\newtheorem{Thm}{Theorem}[section]
\newtheorem{Lem}[Thm]{Lemma}
\newtheorem{Def}[Thm]{Definition}
\newtheorem*{Def*}{Definition}
\newtheorem{Prop}[Thm]{Proposition}
\newtheorem{Conj}[Thm]{Conjecture}
\newtheorem{Cor}[Thm]{Corollary}
\newtheorem{Rmk}[Thm]{Remark}
\newtheorem{eg}[Thm]{Example}
\newtheorem{Question}[Thm]{Question}

\newtheorem*{Claim*}{Claim}

\def\R{\mathbb{R}}
\def\C{\mathbb{C}}
\def\Z{\mathbb{Z}}

\def\P{\mathbb{P}}
\def\H{\mathbb{H}}
\def\T{\mathbb{T}}

\def\D{\mathcal{D}}
\def\O{\mathcal{O}}

\def\NN{\mathcal{N}}
\def\ZZ{\mathcal{Z}}
\def\PP{\mathcal{P}}

\def\cM{\mathcal{M}}

\def\cC{\mathcal{C}}

\def\Stab{\mathrm{Stab}}
\def\Aut{\mathrm{Aut}}
\def\ch{\mathrm{ch}}
\def\td{\mathrm{td}}
\def\Hom{\mathrm{Hom}}

\def\GL{\mathrm{GL}}
\def\NS{\mathrm{NS}}
\def\sys{\mathrm{sys}}
\def\Coh{\mathrm{Coh}}
\def\Fuk{\mathrm{Fuk}}

\def\vol{\mathrm{vol}}

\def\sph{\mathrm{sph}}
\def\disc{\mathrm{disc}}
\def\Quad{\mathrm{Quad}}

\def\ra{\rightarrow}
\def\bs{\backslash}

\title{Systolic inequalities for K3 surfaces via stability conditions}

\author{Yu-Wei Fan}
\date{}

\begin{document}

\maketitle

%%%%% ABSTRACT %%%%%
\begin{abstract}

We introduce the notions of categorical systoles and categorical volumes of Bridgeland stability conditions on triangulated categories. We prove that for any projective K3 surface $X$, there exists a constant $C$ depending only on the rank and discriminant of $NS(X)$, such that
\[
\sys(\sigma)^2\leq C\cdot\vol(\sigma)
\]
holds for any stability condition on $\D^b\Coh(X)$. This is an algebro-geometric generalization of a classical systolic inequality on two-tori. We also discuss applications of this inequality in symplectic geometry. \\


\end{abstract}

\tableofcontents

%%%%% INTRODUCTION %%%%%
\section{Introduction}

Let $(M, g)$ be a Riemannian manifold. Its \emph{systole} $\sys(M,g)$ is defined to be the least length of a non-contractible loop in $M$. In 1949, Charles~Loewner proved that 
\[
\sys(\T^2,g)^2\leq\frac{2}{\sqrt3}\mathrm{vol}(\T^2,g)
\]
holds for any Riemannian metric $g$ on a two-torus $\T^2$. There are various generalizations of Loewner's tours systolic inequality. We refer to \cite{Katz} for a survey on the rich subject of systolic geometry.

The first goal of the present article is to propose a new generalization of Loewner's torus systolic inequality from the perspective of \emph{Calabi--Yau geometry}. We start with an observation in the case of a two-torus. Suppose that the torus is flat $\T^2_\tau\cong\C/\Z+\tau\Z$, and is equipped with the standard complex structure $\Omega=dz$ and symplectic structure $\omega=dx\wedge dy$. Then the shortest non-contractible loops must be straight lines, therefore are \emph{special Lagrangian submanifolds} with respect to the complex and symplectic structures. Under these assumptions, Loewner's torus systolic inequality can be interpreted as:
\begin{equation}
\label{eq:Loewner}
\inf_{\text{sLag\ }L\subset\T^2_\tau}
\Big| \int_L dz \Big|^2 \leq \frac{1}{\sqrt3}
\Big| \int_{\T^2_\tau} dz \wedge d\overline{z} \Big|
\text{\ \  for all }
\tau\in\H.
\end{equation}
The key observation is that the quantities in both sides of this inequality can be generalized to any Calabi--Yau manifold.

We propose the following definition of \emph{systole of a Calabi--Yau manifold}, with respect to its complex and symplectic structures.

\begin{Def}\label{def:sys_lag}
Let $Y$ be a Calabi--Yau manifold, equipping with a symplectic form $\omega$ and a holomorphic top form $\Omega$. Then its \emph{systole} is defined to be
\[
\sys(Y,\omega,\Omega)\coloneqq \inf \Big\{ \Big| \int _L \Omega \Big| \colon L\text{ is a compact special Lagrangian in }(Y,\omega,\Omega)\Big\}.
\]
\end{Def}

With this definition, we propose the following question that naturally generalizes inequality (\ref{eq:Loewner}) to any Calabi--Yau manifold.

\begin{Question}\label{Q:A}
Let $Y$ be a Calabi--Yau manifold and $\omega$ be a symplectic form on $Y$. Does there exist a constant $C=C(Y,\omega)>0$ such that
\[
\sys(Y,\omega,\Omega)^2
 \leq C\cdot  \Big|\int_Y\Omega\wedge\overline\Omega\Big|
\]
holds for any holomorphic top form $\Omega$ on $Y$?
\end{Question}

Here we treat the Calabi--Yau manifolds topologically so that the complex structures $\Omega$ can vary.
Note that the choice of the symplectic structure is not important in the case of two-tori,
since any one-dimensional submanifold is Lagrangian in $\T^2$.
However, in higher dimensions, the notion of Lagrangian submanifolds certainly depends on the choice of the symplectic structure.
Therefore the systolic constant $C$ should depend on the symplectic structure in general, unlike the case of $\T^2$.
Also, note that the ratio $|\int_L\Omega  |^2 / |\int_Y\Omega\wedge\overline\Omega|$
has been considered in the context of \emph{attractor mechanism} in physics
\cite{DRY, KS2, Moo}, which is of independent interest.

The second goal of the present article is to introduce the definitions of categorical systoles and categorical volumes of
Bridgeland stability conditions on triangulated categories.

\begin{Def}[see Definition \ref{def:sys}]\label{def:sys_stab}
Let $\D$ be a triangulated category and $\sigma=(\ZZ,\PP)$ be a Bridgeland stability condition on $\D$.
Its \emph{systole} is defined to be
$$
\sys(\sigma)\coloneqq\min \{ | \ZZ_\sigma(E) | \colon E \text{ is a }\sigma\text{-semistable object in }\D \}.
$$
\end{Def}

Note that in the definition of systole of a Bridgeland stability condition, the minimum can always be attained by some $\sigma$-stable object $E$, therefore we can write ``min" instead of ``inf" (see Remark~\ref{Rmk:sys}).

\begin{Def}[see Definition \ref{vol}]
Let $\{E_i\}$ be a basis of the numerical Grothendieck group $\NN(\D)$, and $\sigma=(\ZZ,\PP)$ be a Bridgeland stability condition on $\D$.
Its \emph{volume} is defined to be
$$
\vol(\sigma)\coloneqq\Big|\sum_{i,j} \chi^{i,j} \ZZ(E_i) \overline{\ZZ(E_j)}\Big|,
$$
where $(\chi^{i,j})=(\chi(E_i,E_j))^{-1}$ is the inverse matrix of the Euler pairings.
\end{Def}

The motivations of these definitions stem from the correspondence between flat surfaces and stability conditions, and the conjectural description of stability conditions on the Fukaya categories of Calabi--Yau manifolds. We refer to Section \ref{sec:sysvol} for more details. Note that under these correspondences, $\sys(\sigma)$ is the categorical generalization of $\sys(Y,\omega,\Omega)$ in Definition \ref{def:sys_lag}, and $\vol(\sigma)$ is the categorical generalization of the holomorphic volume  $\Big|\int_Y\Omega\wedge\overline\Omega\Big|$.

As a sanity check of these definitions, we prove the following categorical analogue of Loewner's torus systolic inequality.

\begin{Thm}[see Theorem \ref{Thm:Ell}]
\label{thm:introell}
Let $\D=\D^b\Coh(E)$ be the derived category of an elliptic curve $E$.
Then
$$
\sys(\sigma)^2 \leq  \frac{1}{\sqrt3} \cdot \vol(\sigma)
$$
holds for any $\sigma\in\Stab(\D)$.
\end{Thm}

We then propose the following algebro-geometric analogue of Question \ref{Q:A},
which is the higher-dimensional generalization of Theorem \ref{thm:introell}.

\begin{Question}\label{Q:B}
Let $X$ be a Calabi--Yau manifold and $\Omega$ be a complex structure on $X$.
Let $\D=\D^b\Coh(X, \Omega)$ be its derived category of coherent sheaves.
Does there exist a constant $C=C(X, \Omega)>0$ such that
$$
\sys(\sigma)^2 \leq C \cdot
\vol(\sigma)
$$
holds for any $\sigma\in\Stab^\dagger(\D)$?
Here $\Stab^\dagger(\D)$ denotes the distinguished connected component of $\Stab(\D)$ that contains geometric stability conditions.
\end{Question}

Note that Questions \ref{Q:A} and \ref{Q:B} are related via the \emph{mirror symmetry conjecture}, which is a conjectural duality between algebraic geometry and symplectic geometry.
The homological mirror symmetry conjecture proposed by Kontsevich \cite{Kon} states that for any Calabi--Yau manifold with a symplectic structure $(Y,\omega)$ , there exists a Calabi--Yau manifold with a complex structure $(X,\Omega)$ such that there is an equivalence between the derived Fukaya category of $Y$ and the bounded derived category of coherent sheaves on $X$:
\[
\D^b\Fuk(Y,\omega)\cong\D^b\Coh(X,\Omega).
\]
It is conjectured by Bridgeland \cite{BriSurvey} and Joyce \cite{Joyce} that a holomorphic top form on $Y$ should give rise to a Bridgeland stability condition $\sigma_\Omega$ on $\D^b\Fuk(Y,\omega)$ (see Conjecture~\ref{conj:BridgelandJoyce}).
The conjectural stability condition satisfies $\sys(\sigma_\Omega)=\sys(Y,\omega,\Omega)$ and $\vol(\sigma_\Omega)=\left|\int_Y\Omega\wedge\overline\Omega\right|$.
Therefore we can consider Question~\ref{Q:B} as the mirror counterpart of Question~\ref{Q:A}.
We refer to Section \ref{sec:future} for more discussions on this.

The third goal, which is the main result of the present article, is to give an affirmative answer to Question \ref{Q:B}
for any complex projective K3 surface. A priori there is no reason to believe that Question \ref{Q:A} and Question \ref{Q:B} have affirmative answers in general. The following theorem is the first evidence that the natural categorical generalization of systolic inequality is possible for higher-dimensional Calabi--Yau manifolds.

\begin{Thm}[see Theorem \ref{mainThm1}]
\label{thm:K3}
Let $X$ be a complex projective K3 surface. Then
\[
\sys(\sigma)^2\leq C\cdot\vol(\sigma)
\]
for any $\sigma\in\Stab^\dagger(\D^b\Coh(X))$, where
\[
C = \frac{((\rho+2)!)^2|\disc\ \NS(X)|}{2^\rho}+4.
\]
Here $\rho$ and $\disc$ denote the rank and the discriminant of the N\'eron--Severi group $\NS(X)$, respectively.
\end{Thm}

Moreover, when the K3 surface is of Picard rank one, we can use a different method to get a better systolic bound.

\begin{Thm}[see Theorem \ref{mainThm2}]
Let $X$ be a K3 surface of Picard rank one, with $\NS(X)=\Z H$ and $H^2=2n$.
Then
$$
\sys(\sigma)^2\leq 4(n+1)\cdot\vol(\sigma)
$$
holds for any $\sigma\in\Stab^\dagger(\D^b\Coh(X))$.
\end{Thm}

Finally, we remark that one can also define a categorical generalization of systole using only the \emph{spherical objects} (Remark \ref{Rmk:Sph}), which we denoted by $\sys_\sph(\sigma)$. However, as we prove in Proposition \ref{Prop:Sph}, the ratio $\sys_\sph(\sigma)^2/\vol(\sigma)$ is unbounded in general.

\subsection*{Related work}
After the first version of the present article was posted online,
Haiden \cite{Hai} proves a systolic inequality for certain higher-dimensional symplectic torus,
and Pacini \cite{Pacini} proposes a higher-dimensional generalization of extremal lengths and complex systolic inequalities.

\subsection*{Organization}
In Section 2, we recall the definition of Bridgeland stability conditions and 
introduce the notions of categorical systole and categorical volume.
In Section 3, we give an affirmative answer to Question \ref{Q:B} for elliptic curves by proving Theorem \ref{thm:introell}.
In Section 4, we give an affirmative answer to Question \ref{Q:B} for any K3 surface by proving Theorem \ref{thm:K3}.
In Section 5, we discuss some directions for future studies.

\subsection*{Acknowledgements}
The author would like to thank Jayadev~Athreya, Simion~Filip, Fabian~Haiden, Pei-Ken~Hung, Atsushi~Kanazawa, Kuan-Wen~Lai, Heather~Lee, Yu-Shen~Lin, Emanuele~Macr\`i, Nikolay~Moshchevitin, Shing-Tung~Yau and Xiaolei~Zhao for helpful discussions and correspondences. This work was partially supported by Simons Collaboration Grant on Homological Mirror Symmetry.
The author would also like to thank an anonymous referee for helpful comments and suggestions.

%%%%% Definitions %%%%%
\section{Categorical systoles and categorical volumes}
\label{sec:sysvol}

\subsection{Bridgeland stability conditions}

In the seminal work \cite{Bri}, Bridgeland introduced the notion of stability conditions on triangulated categories.
We recall the definition and some basic properties of Bridgeland stability conditions.

Throughout the article,
a triangulated category $\D$ is essentially small, linear over $\C$, and is of finite type.
The last condition means that for any pair of objects $E, F\in\D$, the $\C$-vector space
$\oplus_{i\in\Z}\Hom_{\D}(E,F[i])$ is of finite-dimensional.
The Euler form $\chi$ on the Grothendieck group $K(\D)$ is given by the alternating sum
$$
\chi(E,F)\coloneqq\sum_{i}(-1)^i \dim_\C \Hom_{\D}(E,F[i]). 
$$
The numerical Grothendieck group $\NN(\D):=K(\D)/K(\D)^{\perp_{\chi}}$ is defined to be the quotient of $K(\D)$ by the null space of the Euler pairing $\chi$. 
We assume that $\D$ is \emph{numerically finite}, that is, $\NN(\D)$ is of finite rank. 
One large class of examples of such triangulated categories is provided by the bounded derived category of coherent sheaves $\D^b\Coh(X)$ of a smooth projective variety $X$.

%Now we recall the definition of Bridgeland stability conditions on a triangulated category $\D$.
%We will only consider the Bridgeland stability conditions that are \emph{full} and \emph{numerical}.

\begin{Def}[\cite{Bri}]\label{def:stab}
A (full numerical) stability condition $\sigma=(\ZZ,\PP)$ on a triangulated category $\D$ consists of:
	\begin{itemize}
	\item a group homomorphism $\ZZ\colon\NN(\D)\rightarrow\C$, and
	\item a collection of full additive subcategories $\PP=\{\PP(\phi)\}_{\phi \in \R}$ of $\D$,
	\end{itemize}
such that: 
\begin{enumerate}[label=(\alph*)]
\item If $0\neq E \in \PP(\phi)$, then $\ZZ(E)\in\R_{>0}\cdot e^{i\pi \phi}$.
\item $\PP(\phi+1)=\PP(\phi)[1]$. 
\item If $\phi_1>\phi_2$ and $A_i \in \PP(\phi_i)$, then $\Hom(A_1,A_2)=0$. 
\item For every $0 \ne E \in \mathcal{D}$, there exists a (unique) collection of distinguished triangles
$$
\xymatrix{
0=E_0 \ar[r] & E_1\ar[d]  \ar[r] & E_2 \ar[r] \ar[d]& \cdots \ar[r]& E_{k-1}\ar[r] & E \ar[d]\\
                    & B_1 \ar@{-->}[lu]& B_2 \ar@{-->}[lu] &  & & B_k  \ar@{-->}[lu] 
}
$$
such that $B_i \in \PP(\phi_i)$ and $\phi_1>\phi_2>\cdots>\phi_k$.
Denote $\phi^+_\sigma(E)\coloneqq\phi_1$ and $\phi^-_\sigma(E)\coloneqq\phi_k$.
The \emph{mass} of $E$ is defined to be $m_\sigma(E)\coloneqq\sum_i |\ZZ(B_i)|$.
\item\label{def:supportprop}
(Support property \cite{KS}) There exists a constant $C>0$ and a norm $||\cdot||$
on $\NN(\D)\otimes_\Z \R$ such that
$$
||E||\leq C|\ZZ(E)|
$$
for any semistable object $E$.
\end{enumerate}
\end{Def}

The group homomorphism $\ZZ$ is called the \emph{central charge}, and the nonzero objects in $\PP(\phi)$ are called the \emph{semistable objects of phase $\phi$}.
The additive subcategories $\PP(\phi)$ actually are abelian,
and the simple objects of $\PP(\phi)$ are said to be \emph{stable}.

The space of (full numerical) Bridgeland stability conditions on $\D$ is denoted by $\Stab(\D)$.
There is a nice topology on $\Stab(\D)$ introduced by Bridgeland,
which is induced by the generalized distance:
$$
d(\sigma_1,\sigma_2)=\sup_{0\neq E\in\D}\Big\{|\phi^-_{\sigma_2}(E)-\phi^-_{\sigma_1}(E)|, |\phi^+_{\sigma_2}(E)-\phi^+_{\sigma_1}(E)|, |\log\frac{m_{\sigma_2}(E)}{m_{\sigma_1}(E)}|\Big\}\in[0,\infty].
$$

The forgetful map
$$
\Stab(\D)\longrightarrow \Hom(\NN(\D),\C), \ \ \ \sigma=(\ZZ,\PP) \mapsto \ZZ
$$
is a local homeomorphism  \cite{Bri, KS}.
Hence $\Stab(\D)$ is a complex manifold.

There are two natural group actions on the space of Bridgeland stability conditions $\Stab(\D)$ which commute with each other \cite[Lemma 8.2]{Bri}. Firstly, the group of autoequivalences $\Aut(\D)$ acts on $\Stab(\D)$ as isometries with respect to the generalized metric: Let $\Phi\in\Aut(\D)$ be an autoequivalence, define
\[
\sigma=(\ZZ, \PP) \mapsto \Phi\cdot\sigma\coloneqq(\ZZ\circ[\Phi]^{-1}, \PP'),
\]
where $[\Phi]$ is the induced automorphism on $\NN(\D)$, and $\PP'(\phi)\coloneqq\Phi(\PP(\phi))$.

Secondly, the universal cover $\widetilde{\GL^+(2, \R)}$ also admits a natural group action on $\Stab(\D)$. Recall that $\widetilde{\GL^+(2, \R)}$ is isomorphic to the group of pairs $(T,f)$, where $T\in\GL^+(2,\R)$ and $f:\R\ra\R$ is an increasing map with $f(\phi+1)=f(\phi)+1$, such that their induced maps on $(\R^2\bs\{0\})/\R_{>0}\cong \R/2\Z \cong S^1$ coincide. Let $g=(T,f)\in\widetilde{\GL^+(2, \R)}$, define
\[
\sigma=(\ZZ, \PP) \mapsto \sigma\cdot g \coloneqq (T^{-1}\circ \ZZ, \PP''),
\]
where $\PP''(\phi)\coloneqq\PP(f(\phi))$. Note that the subgroup $\C\subset\widetilde{\GL^+(2, \R)}$ acts freely and transitively on $\Stab(\D)$. Let $z\in\C$, then its action on $\Stab(\D)$ is
\[
\sigma=(\ZZ, \PP) \mapsto \sigma\cdot z\coloneqq (\exp(-i\pi z)\ZZ, \PP'''),
\]
where $\PP'''(\phi)\coloneqq\PP(\phi+\mathrm{Re}(z))$.

%%%%% Cat sys %%%%%
\subsection{Categorical systoles}
We first recall some results and conjectures that motivate our definition of categorical systoles.

In a striking series of work by Gaiotto--Moore--Neitzke \cite{GMN}, Bridgeland--Smith \cite{BS} and Haiden--Katzarkov--Kontsevich \cite{HKK}, the connections between stability conditions and Teichm\"uller theory have been established. One of the main results in this direction is the following theorem.

\begin{Thm}[\cite{HKK}, Theorem 5.2 and 5.3]
\label{thm:HKK}
Let $S$ be a marked surface of finite type, $\cM(S)$ be the space of marked flat structures on $S$, and $\Fuk(S)$ be the Fukaya category of $S$. Then there is a natural map
\[
\cM(S)\ra\Stab(\Fuk(S))
\]
which is bianalytic onto its image, which is a union of connected components. Under this bianalytic map,
\begin{itemize}
\item saddle connections on flat surfaces are corresponded to semistable objects,
\item their lengths are corresponded to the absolute values of central charges. 
\end{itemize}
\end{Thm}

Recall that the systole of a flat surface is defined to be the length of its shortest saddle connection. Based on the correspondence established in Theorem \ref{thm:HKK},
it is natural to define the systole of a Bridgeland stability condition to be the smallest absolute value of central charge of semistable objects.

Another important source of motivation for defining categorical systole is a conjectural description of stability conditions on the Fukaya categories of Calabi--Yau manifolds, proposed by Bridgeland and Joyce.

\begin{Conj}[\cite{BriSurvey}, Section 2.4 and \cite{Joyce}, Conjecture 3.2]
\label{conj:BridgelandJoyce}
Let $Y$ be a Calabi--Yau manifold equipping with a symplectic form $\omega$, and $\D^\pi\Fuk(Y,\omega)$ be the derived Fukaya category of $Y$. For any holomorphic top form $\Omega$ on $Y$, there exists a natural Bridgeland stability condition $\sigma_{\Omega}$ on $\D^\pi\Fuk(Y,\omega)$, such that the central charges are given by the period integrals along Lagrangians
\[
\ZZ_\Omega(L)=\int_L\Omega,
\]
and the $\sigma_\Omega$-semistable objects are given by compact special Lagrangian submanifolds with respect to $\Omega$ and $\omega$.
\end{Conj}

If we assume this conjecture to be true, then the definition of systole of $(Y,\omega,\Omega)$ in Definition \ref{def:sys_lag} can be written as
\begin{align*}
\sys(Y,\omega,\Omega) &\coloneqq \inf \left\{ \Big| \int _L \Omega \Big| \colon L\text{ is a compact special Lagrangian in }(Y,\omega,\Omega)\right\}\\
&=\inf\left\{|\ZZ(L)|\colon L\text{ is a }\sigma_\Omega\text{-semistable object in} \D^\pi\Fuk(Y,\omega)\right\}
\end{align*}

Motivating from the above discussions, we propose the following definition of systole of a Bridgeland stability condition.

\begin{Def}\label{def:sys}
Let $\sigma$ be a Bridgeland stability condition on a triangulated category $\D$.
Its \emph{systole} is defined to be
$$
\sys(\sigma)\coloneqq\min \{ | \ZZ_\sigma(E) | \colon E \text{ is a }\sigma\text{-semistable object in }\D \}.
$$
\end{Def}

\begin{Rmk}\label{Rmk:sys}
Note that we can write ``min" instead of ``inf" in the definition of categorical systole for the following reason.
For any $R>0$, consider the following two subsets of $\NN(\D)$:
$$
S_R^{(1)}\coloneqq\{v\in\NN(\D)\colon\text{there exists }\sigma\text{-semistable }E\text{ such that } [E]=v \text{ and } |\ZZ_\sigma(v)|<R\},
$$
$$
S_R^{(2)}\coloneqq\{v\in\NN(\D)\colon\text{there exists }\sigma\text{-semistable }E\text{ such that } [E]=v \text{ and } ||v||<CR\},
$$
where $C>0$ and $||\cdot||$ are the constant and the norm on $\NN(\D)\otimes_\Z\R$ appeared in the support property of stability conditions, see Definition~\ref{def:stab}\ref{def:supportprop}.
By the support property, we have $S_R^{(1)}\subset S_R^{(2)}$.
The set $S_R^{(2)}$ is finite since it is a subset of $\{v\in\NN(\D)\colon||v||<CR\}$ which is finite.
Therefore $S_R^{(1)}$ is a finite set, which implies that the minimum of
\[
 \{ | \ZZ_\sigma(E) | \colon E \text{ is a }\sigma\text{-semistable object in }\D \}
\]
in the definition of systole must be attained by some $\sigma$-semistable object $E$.
In fact, the systole must be attained by some $\sigma$-\emph{stable} object: if $E$ is a $\sigma$-semistable object and not a $\sigma$-stable object, then it is clear that the absolute value of the central charge of any non-trivial $\sigma$-stable factor of $E$ is strictly less than $|Z_{\sigma}(E)|$.
\end{Rmk}

\begin{Rmk}
In the following table, we summarize the correspondences between the systoles of flat surfaces, Calabi--Yau manifolds, and stability conditions discussed in Theorem~\ref{thm:HKK} and Conjecture~\ref{conj:BridgelandJoyce}.
\begin{center}
 \begin{tabular}{||c  |  c  |  c ||} 
 \hline
Surface $S$ & Calabi--Yau $(Y, \omega)$ & Triangulated category  \\
 \hline\hline
abelian differentials & holomorphic top forms & stability conditions  \\
 \hline
saddle connections & special Lagrangians & semistable objects \\
 \hline 
lengths & period integrals & central charges \\
 \hline
$\sys(S)$ & $\sys(Y, \omega, \Omega)$ & $\sys(\sigma)$ \\
 \hline
  & $|\int_Y\Omega\wedge\overline\Omega|$ & $\vol(\sigma)$ \\
 \hline
\end{tabular}
\end{center}
The definition of categorical volume $\vol(\sigma)$, which is also inspired by Conjecture \ref{conj:BridgelandJoyce}, will be discussed in the next subsection.
\end{Rmk}

Now we study how the categorical systole changes under the natural group actions on $\Stab(\D)$.

\begin{Lem}
\label{Lem:sys}
Let $\sigma$ be a Bridgeland stability condition on $\D$. Then
	\begin{enumerate}[label=(\alph*)]
	\item
	\begin{align*}
	\sys(\sigma) & = \min\{|\ZZ_\sigma(E)| \colon E \text{ is a } \sigma\text{-stable object in } \D \} \\
	& = \min \{ m_\sigma(E) \colon 0\neq E\in\D \}
	\end{align*}
	Recall that $m_\sigma(E)$ is the mass of $E$ with respect to $\sigma$ (Definition \ref{def:stab}).
	
	\item
	For any autoequivalence $\Phi\in\Aut(\D)$, $\sys(\Phi\cdot\sigma) = \sys(\sigma)$.
	
	\item
	For any complex number $z=x+iy\in\C$,
	$\sys(\sigma\cdot z) = \exp(y\pi)\cdot \sys(\sigma)$.
	
	\item
	For any $g=(T,f)\in\widetilde{\GL^+(2, \R)}$, write $T^{-1}=\begin{bmatrix}a&b\\c&d\end{bmatrix}\in\GL^+(2,\R)$,
	then
	\[
	\sys(\sigma\cdot g) \leq |t_1|(1+|t_2|)\cdot\sys(\sigma),
	\]
	where $t_1=\frac{(a+d)+i(c-b)}{2}\neq0$ and $t_2=\frac{(a-d)+i(b+c)}{(a+d)+i(c-b)}$. Note that $|t_2|<1$.
	
	\end{enumerate}
\end{Lem}

\begin{proof}
The first three statements follow easily from the definition. To prove the last statement, we use a similar idea in the proof of \cite[Theorem 4.3]{Hai}. One can verify by direct computation that
\begin{equation}
\label{eq:GL2action}
\ZZ_{\sigma\cdot g}(E) = t_1 (\ZZ_\sigma(E) + t_2 \overline{\ZZ_\sigma(E)})
\end{equation}
holds for any $\sigma$ and $E$.
By Remark \ref{Rmk:sys}, there exists a $\sigma$-stable object $E$ such that $\sys(\sigma)=|\ZZ_\sigma(E)|$.
Since the actions by $\widetilde{\GL^+(2, \R)}$ do not change the set of stable objects, $E$ is also $\sigma\cdot g$-stable. Therefore,
\begin{align*}
|t_1|(1+|t_2|)\cdot\sys(\sigma) & = |t_1|(1+|t_2|)\cdot|\ZZ_\sigma(E)| \\
& \geq |\ZZ_{\sigma\cdot g}(E)| \geq \sys(\sigma\cdot g)
\end{align*}
\end{proof}

\begin{Prop}
The function 
\[
\sys\colon\Stab(\D)\longrightarrow \R_{>0}, \ \ \ \sigma \mapsto \sys(\sigma)
\]
is continuous on $\Stab(\D)$.
\end{Prop}

\begin{proof}
Let $\sigma_1, \sigma_2\in\Stab(\D)$ be two Bridgeland stability conditions with $d(\sigma_1, \sigma_2) < \epsilon$. By Remark \ref{Rmk:sys}, $\sys(\sigma_1)=m_{\sigma_1}(E)$ for some object $E\in\D$.
Thus
$$
\log\sys(\sigma_1) =
\log m_{\sigma_1}(E) >
\log m_{\sigma_2}(E) - \epsilon \geq
\log\sys(\sigma_2) - \epsilon.
$$
Similarly, we have $\log\sys(\sigma_2)>\log\sys(\sigma_1)-\epsilon$.
Hence $\log\sys(\sigma)$ is a continuous function on $\Stab(\D)$, and so is $\sys(\sigma)$.
\end{proof}

Let us consider some basic examples of categorical systoles of algebraic stability conditions.
In Section \ref{sec:ellcurve} and Section \ref{sec:K3}, we will be dealing with non-algebraic stability conditions on derived categories of coherent sheaves.

\begin{eg}[Derived category of representations of $A_2$-quiver] \label{eg:A2}
Let $Q$ be the $A_2$-quiver $(\boldsymbol\cdot \ra \boldsymbol\cdot)$ and
$\D=\D^b(\mathrm{Rep}(Q))$ be the bounded derived category of the category of representations of $Q$.
Let $E_1$ and $E_2$ be the simple objects in $\mathrm{Rep}(Q)$
with dimension vectors $\underline\dim(E_1)=(1, 0)$ and
$\underline\dim(E_2)=(0, 1)$.
There is one more indecomposable object $E_3$ that fits into the exact sequence
$$
0 \ra E_2 \ra E_3 \ra E_1 \ra 0.
$$

Let $\sigma=(\ZZ, \PP)$ be a stability condition on $\D$ with $z_1:=\ZZ(E_1)$ and
$z_2:=\ZZ(E_2)$.
Suppose that $\PP(0,1] = \mathrm{Rep}(Q)$. Then

\begin{itemize}
\item When $\arg(z_1) < \arg(z_2)$, the only $\sigma$-stable objects are $E_1$ and $E_2$ up to shiftings. Thus $\sys(\sigma)=\min\{|z_1|, |z_2|\}$.
\item When $\arg(z_1) > \arg(z_2)$, the only $\sigma$-stable objects are $E_1$, $E_2$ and $E_3$ up to shiftings. Thus $\sys(\sigma)=\min\{|z_1|, |z_2|, |z_1+z_2|\}$.
\end{itemize}

Therefore, in order to compute the categorical systoles of stability conditions in different chambers,
one needs to compute the central charges of different sets of dimensional vectors.
Note that this is not the case for the derived categories of elliptic curves and K3 surfaces,
where the categorical systole of any stability condition is the minimum of absolute values of central charges of the \emph{same} set of classes in $\NN(\D)$ (see Section \ref{sec:ellcurve} and Proposition~\ref{Prop:key}).
\end{eg}

\begin{eg}[3-Calabi--Yau categories arising from marked Riemann surfaces]
There is a 3-Calabi--Yau category $\D_{S,M}$ (up to equivalence) associated to each marked Riemann surface $(S,M)$, see for instance \cite[Section 9]{BS}. The main theorem in \cite{BS} states that
\begin{equation}
\label{eq:BS}
\Stab(\D_{S,M})/\Aut(\D_{S,M}) \cong \Quad(S,M),
\end{equation}
where $\Quad(S,M)$ is the moduli space of certain meromorphic quadratic differentials on $S$ with simple zeros.
By Lemma \ref{Lem:sys} (2), the function $\sys$ on $\Stab(\D)$ given by categorical systole descends to a function on the quotient $\Stab(\D)/\Aut(\D)$. Under the equivalence (\ref{eq:BS}), the categorical systole of a stability condition on $\D_{S,M}$ is given by the minimum among the period integrals $\int \sqrt\phi$ along saddle connections of the corresponding quadratic differential.
\end{eg}

%%%%% Cat vol %%%%%
\subsection{Categorical volumes}
We recall the notion of \emph{categorical volumes} of Bridgeland stability conditions introduced in \cite{FKY}. It is the categorical analogue of the holomorphic volume $\Big|\int_Y\Omega\wedge\overline\Omega\Big|$ of a compact Calabi--Yau manifold $Y$ with holomorphic top form $\Omega$.

Let $Y$ be a compact Calabi--Yau manifold of dimension $n$, and $\{A_i\}$ be a basis of the torsion-free part of
$H_n(X,\Z)$. Then one can rewrite the holomorphic volume of $Y$ as
$$
\Big|\int_Y\Omega\wedge\overline\Omega\Big|=\Big|\sum_{i,j}\gamma^{i,j}\int_{A_i}\Omega\int_{A_j}\overline\Omega\Big|,
$$
where $(\gamma^{i,j})=(A_i\cdot A_j)^{-1}$ is the inverse matrix of the intersection pairings.

Recall that in Conjecture \ref{conj:BridgelandJoyce}, the period integral $\int\Omega$ should give the central charge of a Bridgeland stability condition on the derived Fukaya category $\D^\pi\Fuk(Y)$ that is associated to the holomorphic top form $\Omega$. This motivates the following definition.

\begin{Def}[\cite{FKY}]\label{vol}
Let $\{E_i\}$ be a basis of the numerical Grothendieck group $\NN(\D)$
and let $\sigma=(\ZZ,\PP)$ be a Bridgeland stability condition on $\D$.
Its \emph{volume} is defined to be
$$
\vol(\sigma)\coloneqq\Big|\sum_{i,j} \chi^{i,j} \ZZ(E_i) \overline{\ZZ(E_j)}\Big|,
$$
where $(\chi^{i,j})=(\chi(E_i,E_j))^{-1}$ is the inverse matrix of the Euler pairings.
\end{Def}

One can easily check that the above definition is independent of the choice of the basis $\{E_i\}$: Suppose $\{F_j\}$ is another basis of $\NN(\D)$, where $A$ is the unimodular matrix that relates the bases $\{E_i\}$ and $\{F_j\}$, i.e.~$F_j=\sum_kA_{jk}E_k$. Denote the Gram matrices $(\chi(E_i,E_j))$ and $(\chi(F_i,F_j))$ by $\chi_E$ and $\chi_F$, respectively. Then we have
$\chi_F=A\chi_EA^T$, and therefore
\begin{align*}
\sum_{i,j}(\chi_F)^{-1}_{ij}\ZZ(F_i)\overline{\ZZ(F_j)} &
= \sum_{i,j}(A\chi_EA^T)^{-1}_{ij}\ZZ(F_i)\overline{\ZZ(F_j)} \\
& =\sum_{i,j,k,\ell} (A^T)^{-1}_{ik}(\chi_E)^{-1}_{k\ell}A^{-1}_{\ell j}\ZZ(F_i)\overline{\ZZ(F_j)} \\
&=\sum_{i,j,k,\ell,m,n} A^{-1}_{ki}(\chi_E)^{-1}_{k\ell}A^{-1}_{\ell j}A_{im}A_{jn}\ZZ(E_m)\overline{\ZZ(E_n)}\\
&=\sum_{k,\ell,m,n}\delta_{km}\delta_{\ell n}(\chi_E)^{-1}_{k\ell}\ZZ(E_m)\overline{\ZZ(E_n)}\\
&=\sum_{m,n}(\chi_E)^{-1}_{mn}\ZZ(E_m)\overline{\ZZ(E_n)}.
\end{align*}

It is important to note that unlike the categorical systoles, the categorical volume of a stability condition $\sigma=(\ZZ, \PP)$ depends only on its central charge. The proof of the following lemma is straightforward.

\begin{Lem}\label{Lem:vol}
Let $\sigma$ be a Bridgeland stability condition on $\D$. Then
	\begin{enumerate}[label=(\alph*)]
	\item For any autoequivalence $\Phi\in\Aut(\D)$, $\vol(\Phi\cdot\sigma) = \vol(\sigma)$.
	\item For any complex number $z=x+iy\in\C$, $\vol(\sigma\cdot z) = \exp(2y\pi)\cdot \vol(\sigma)$.
	\end{enumerate}
\end{Lem}

Now we recall some computations of categorical volumes in \cite{FKY} which will be used in the later sections. 

\begin{eg}[Elliptic curves] \label{eg:ell}
Let $\D=\D^b\Coh(E)$ be the derived category of coherent sheaves on an elliptic curve $E$.
Let $\sigma=(\ZZ,\PP)$ be a stability condition on $\D$ with central charge
$$
\ZZ_\sigma(F) = - \mathrm{deg}(F) + (\beta + i \omega) \cdot \mathrm{rk}(F),
$$
where $\beta\in\R$ and $\omega>0$.
Choose $\{\O_x, \O_E\}$ as a basis of the numerical Grothendieck group $\NN(\D)$, where $\O_x$ is a skyscraper sheaf and $\O_E$ is the structure sheaf.
Then the categorical volume of $\sigma$ is
\begin{align}
\vol(\sigma) & = |\ZZ_\sigma(\O_x)\overline{\ZZ_\sigma(\O_E)}-\ZZ_\sigma(\O_E)\overline{\ZZ_\sigma(\O_x)}| \notag \\
 & =  2\omega >0. \notag 
\end{align}

\begin{Rmk}
Let $C$ be a curve of genus $g\geq1$. Then $\Stab(\D^b\Coh(C))=\C\times\H$, and the central charge of any stability condition is of the form
$$
\ZZ_\sigma(F) = - \mathrm{deg}(F) + (\beta + i \omega) \cdot \mathrm{rk}(F) \ \ \  (\beta\in\R\text{  and } \omega>0)
$$
up to the free $\C$-action (see \cite{Bri,Macri}). However, the categorical volume $\vol(\sigma)$ is not given by $2\omega$ unless $g=1$, due to the fact that the Euler pairing on $\NN(\D)$ is not (anti)symmetric if $C$ is not Calabi--Yau. In fact, choose $\{E_1=\O,E_2=\O(1)\}$ as a basis of $\NN(\D)$, then the Gram matrix of Euler pairings is
$
A\coloneqq\begin{bmatrix}
1-g&2-g\\
-g&1-g
\end{bmatrix}
$
by Riemann--Roch. Therefore we have
\begin{align*}
\vol(\sigma) &= \Big|\sum_{1\leq i,j\leq 2} (A^{-1})_{ij}\ZZ(E_i)\overline{\ZZ(E_j)}\Big|\\
&=\Big|(1-g)|\beta+i\omega|^2+(g-2)(\beta+i\omega)\overline{\beta+i\omega-1}\\
&\quad\quad + g(\beta+i\omega-1)\overline{\beta+i\omega}+(1-g)|\beta+i\omega-1|^2\Big|\\
&=\left| 2i\omega+(1-g)\right|\\
&=\sqrt{(g-1)^2+(2\omega)^2}>0.
\end{align*}

On the other hand, there exist stability conditions on $\D^b\Coh(\P^1)$ that are not of this form (non-geometric). It is possible for such stability conditions to have zero categorical volume. For instance, by \cite[Proposition~3.3]{Okada}, there exist stability conditions $\sigma\in\Stab(\D^b\Coh(\P^1))$ such that $\ZZ_\sigma(\O)=\ZZ_\sigma(\O(1))=z\in\C\backslash\{0\}$. Then by the same computation as above, we have
\begin{align*}
\vol(\sigma)=\left| |z|^2-2|z|^2+|z|^2 \right|=0.
\end{align*}
\end{Rmk}

\end{eg}

\begin{eg}[K3 surfaces] \label{eg:K3}
Let $\D=\D^b\Coh(X)$ be the derived category of coherent sheaves on a K3 surface $X$.
Let $\sigma=(\ZZ,\PP)$ be a stability condition on $\D$ with central charge
$$
\ZZ_\sigma(v) = (\exp(\beta+i\omega), v),
$$
where $(\cdot,\cdot)$ is the Mukai pairing on the numerical Grothendieck group $\NN(\D)$,
and $\beta,\omega\in\NS(X)\otimes\R$ with $\omega^2>0$
(c.f.~Section 4).
Let $\{ v_i \}$ be a basis of $\NN(\D)$.
Then the categorical volume of $\sigma$ is
\begin{align*}
\vol(\sigma) & = \Big|\sum_{i,j} \chi^{i,j} \ZZ_\sigma(v_i) \overline{\ZZ_\sigma(v_j)}\Big| \notag \\
 & = \Big|\sum_{i,j} \chi^{i,j} \cdot  (\exp(\beta+i\omega), v_i) \cdot  (\exp(\beta-i\omega), v_j)\Big|  \notag \\
  & = \Big|\sum_{i,j}  (\exp(\beta+i\omega), v_i) \cdot \chi^{i,j} \cdot  (v_j, \exp(\beta-i\omega))\Big|  \ \ (\mathrm{Serre\ duality}) \notag \\
  & = \Big| \Big(\exp(\beta+i\omega), \exp(\beta-i\omega) \Big) \Big|
   		\ \ (\mathrm{compatibility\ btw\ Euler/Mukai\ pairings})\notag \\
  & = 2 \omega^2 > 0. \notag
\end{align*}

Note that the same computation shows that $\sum_{i,j} \chi^{i,j} \ZZ_\sigma(v_i) \ZZ_\sigma(v_j)=0$.

For any $g=(T,f)\in\widetilde{\GL^+(2, \R)}$, one can compute the categorical volume of the stability condition $\sigma\cdot g$ following the same idea in the proof of Lemma \ref{Lem:sys} (4).
Write $T^{-1}=\begin{bmatrix}a&b\\c&d\end{bmatrix}\in\GL^+(2,\R)$.
Define $t_1=\frac{(a+d)+i(c-b)}{2}\neq0$ and $t_2=\frac{(a-d)+i(b+c)}{(a+d)+i(c-b)}$.
By Equation (\ref{eq:GL2action}), we have
\begin{align*}
\vol(\sigma\cdot g) & = \Big|\sum_{i,j} \chi^{i,j} \ZZ_{\sigma\cdot g}(v_i) \overline{\ZZ_{\sigma\cdot g}(v_j)}\Big| \notag \\
&=|t_1|^2 \Big|\sum_{i,j} \chi^{i,j} (\ZZ_\sigma(v_i)+t_2\overline{\ZZ_\sigma(v_i)})
(\overline{\ZZ_\sigma(v_j)}+\overline{t_2}\ZZ_\sigma(v_j)) \Big| \notag \\
&=|t_1|^2(1+|t_2|^2)\cdot\vol(\sigma).
\end{align*}
We use $\chi^{i,j}=\chi^{j,i}$ in the last step, which follows from the fact that $\D$ is a $2$-Calabi--Yau category therefore its Euler pairing is symmetric.
\end{eg}

We should note that although the categorical volume can be defined for stability conditions on any triangulated category $\D$, its geometric meaning is not clear unless $\D$ comes from compact Calabi--Yau geometry. Below is an example of a Calabi--Yau triangulated category for which the categorical volume vanishes for some stability conditions (on the wall of marginal stability conditions).

\begin{eg}[$3$-Calabi--Yau category of the $A_2$-quiver]
Let $\D$ be the $3$-Calabi--Yau category constructed from the Ginzburg 3-Calabi--Yau dg-algebra associated to the $A_2$-quiver \cite{Gin, Kel}.
%Let $Q$ be the $A_2$-quiver $(\boldsymbol\cdot \ra \boldsymbol\cdot)$
%and $\Gamma Q$ be the Ginzburg Calabi--Yau--3 dg-algebra associated with $Q$ \cite{Gin, Kel}.
%Let $\D(\Gamma Q)$ be the derived category of dg-modules over $\Gamma Q$
%and $\D=\D_{\mathrm{fd}}(\Gamma Q)$ be the full subcategory of $\D(\Gamma Q)$
%consisting of dg-modules with finite dimensional total cohomology.
%By Keller \cite[Theorem 6.3]{Kel}, the category $\D$ is a Calabi--Yau--3 category.
%there is a natural isomorphism 
%$
%\Hom(E, F) \cong \Hom (F, E[3])^*
%$
%for any $E, F\in\D$.
%By Smith \cite{Smi}, there is an embedding of $\D$ into the Fukaya category of certain quasi-projective Calabi--Yau threefold.
The numerical Grothendieck group $\NN(\D)$ is generated by two spherical objects $S_1, S_2$.
%each associated with a vertex in the $A_2$-quiver.
%The spherical objects satisfy:
%\begin{itemize}
%\item $\Hom^*_\D(S_1, S_1)=\Hom^*_\D(S_2, S_2)=\C\oplus\C[-3]$.
%\item $\Hom^*_\D(S_1, S_2)=\C[-1]$, $\Hom^*_\D(S_2, S_1)=\C[-2]$.
%\end{itemize}

Let $\sigma=(\ZZ, \PP)$ be a stability condition on $\D$ with $z_1=\ZZ(S_1)$ and $z_2=\ZZ(S_2)$.
Then its categorical volume is
$$
\vol(\sigma)  = |z_1 \overline{z_2} - z_2 \overline{z_1} |
= 2 | \mathrm{Im}(z_1 \overline{z_2}) |,
$$
which vanishes if $z_1 \overline{z_2} \in \R$.
This can happen if $z_1$ and $z_2$ are of the same phase, i.e.,
the stability condition $\sigma$ sits on a \emph{wall} in $\Stab(\D)$.
\end{eg}

%\begin{eg}[Derived category of Kronecker quiver]
%
%Let $K(l)$ be the Kronecker quiver 
%with $l$ parallel arrows between the two vertices.
%Let $\D=\D^b(\mathrm{Rep}(K(l)))$
%be the bounded derived category of the category of representations of $K(l)$.
%
%Let $E_1$ and $E_2$ be the simple objects in $\mathrm{Rep}(K(l))$
%with dimension vectors $\underline\dim(E_1)=(1, 0)$ and
%$\underline\dim(E_2)=(0, 1)$.
%Then
%\begin{itemize}
%\item $\Hom^*_\D(E_1, E_1)=\Hom^*_\D(E_2, E_2)=\C$.
%\item $\Hom^*_\D(E_1, E_2)=\C^l$, $\Hom^*_\D(E_2, E_1)=0$.
%\end{itemize}
%
%Let $\sigma=(\ZZ, \PP)$ be a stability condition on $\D$ with $z_1:=\ZZ(E_1)$ and
%$z_2:=\ZZ(E_2)$.
%Then its categorical volume is
%$$
%\vol(\sigma) = \Big| |z_1|^2 + |z_2|^2 - l z_1 \overline{z_2} \Big|.
%$$
%For $l\geq2$,
%the categorical volume $\vol(\sigma)$ vanishes if
%$z_2=\frac{l\pm\sqrt{l^2-4}}{2} z_1$,
%which again happens when $\sigma$ sits on a wall of stability conditions.
%Note that this example includes the case of derived category of coherent sheaves on projective line
%since $\D^b(\P^1) \cong \D^b(K(2))$.
%
%\end{eg}

%%%%% Elliptic curves %%%%%
\section{Systolic inequality for elliptic curves: a sanity check}
\label{sec:ellcurve}

In this section, we give an affirmative answer to Question \ref{Q:B} in the case of elliptic curves. This provides a sanity check for our definitions of categorical systole and categorical volume.

\begin{Thm}\label{Thm:Ell}
Let $\D=\D^b\Coh(E)$ be the derived category of an elliptic curve $E$.
Then
$$
\sys(\sigma)^2 \leq  \frac{1}{\sqrt3} \cdot \vol(\sigma)
$$
holds for any $\sigma\in\Stab(\D)$.
\end{Thm}

One can consider this inequality as the \emph{mirror} of Loewner's torus systolic inequality in the introduction. We refer to Section \ref{sec:future} for more discussions related to mirror symmetry.

%and the double quotient
%$$
%\Aut(\D)\backslash\Stab(\D)/\C \cong \mathrm{PSL}(2,\Z)\backslash\H
%$$
%is indeed the K\"ahler moduli space of elliptic curve.
%Thus we should take $\Stab^*(\D)=\Stab(\D)$ in Question \ref{Q:B} to be the whole space of stability conditions.

\begin{proof}
By \cite[Theorem 9.1]{Bri}, the $\widetilde{\mathrm{GL}^+(2;\R)}$-action on the space of Bridgeland stability conditions $\Stab(\D)$ of an elliptic curve is free and transitive. Therefore
\[
\Stab(\D)\cong\widetilde{\mathrm{GL}^+(2;\R)}\cong\C\times\H.
\]

By Lemma \ref{Lem:sys} and Lemma \ref{Lem:vol}, the systolic ratio
\[
\frac{\sys(\sigma)^2}{\vol(\sigma)}
\]
is invariant under the free $\C$-action on the space of stability conditions.
Hence we only need to compute the ratio on the quotient space $\Stab(\D)/\C\cong\H$.
The quotient space $\Stab(\D)/\C\cong\H$ can be parametrized by the \emph{normalized stability conditions} as follows.
Let $\tau = \beta + i\omega\in\H$ where $\beta\in\R$ and $\omega>0$.
The associated normalized stability condition $\sigma_\tau$ is given by:
\begin{itemize}
\item Central charge: $\ZZ_\tau(F) = - \mathrm{deg}(F) + \tau\cdot \mathrm{rk}(F)$.
\item For $0<\phi\leq1$, the (semi)stable objects $\PP_\tau(\phi)$ are
	the slope-(semi)stable coherent sheaves whose central charges lie in the ray $\R_{>0}\cdot e^{i\pi\phi}$.
\item For other $\phi\in\R$, define $\PP_\tau(\phi)$ by the property
	$\PP_\tau(\phi+1)=\PP_\tau(\phi)[1]$.
\end{itemize}

Note that there is no wall-crossing phenomenon in the elliptic curve case, i.e., the set of all Bridgeland (semi)stable objects is the same for any stability condition. This makes the computation of categorical systole easier.

To compute the systole of $\sigma_\tau$, by Lemma \ref{Lem:sys} (1), we need to know the central charges of all the stable objects of $\sigma_\tau$,
which are the slope-stable coherent sheaves.
Recall that if $F$ is a slope-stable coherent sheaf on $E$, then it is either a vector bundle or a torsion sheaf.
The slope-stable vector bundles on an elliptic curve $E$ are well-understood,
see for instance \cite{Atiyah, Pol}.
In particular, we have the following facts:
\begin{itemize}
\item Let $F$ be an indecomposable vector bundle of rank $r$ and degree $d$ on an
elliptic curve $E$. Then $F$ is slope-stable if and only if $d$ and $r$ are relatively
prime.
\item Fix a point $x\in E$. For every rational number $\mu=\frac{d}{r}$, where $r>0$ and $(d,r)=1$, there exists a unique slope-stable vector bundle $V_\mu$ of rank $r$ and $\mathrm{det}(V_\mu)\cong\O_E(dx)$.
\end{itemize}

Hence the categorical systole is
\begin{align*}
\sys(\sigma_\tau) &= \min \Big\{ 1, |-d+\tau r| \colon (d,r)=1 \text{ and } r>0 \Big\} \\
&= \Big\{ |-d+\tau r| \colon (d,r)\in\Z^2\bs\{(0,0)\} \Big\} \\
&=\lambda_1(L_\tau),
\end{align*}
where $\lambda_1(L_\tau)$ denotes the least length of a nonzero element in the lattice $L_\tau = \langle 1, \tau \rangle$.

On the other hand, the categorical volume of $\sigma_\tau$ has been computed in Example \ref{eg:ell}, which is equals to $2\omega$.
Thus
$$
\sup_{\tau\in\H}\frac{\sys(\sigma_\tau)^2}{\vol(\sigma_\tau)}
= \frac{1}{2}\cdot\sup_{\tau\in\H}\frac{\lambda_1(L_\lambda)^2}{\omega}.
$$

Note that $\omega$ is the area of the parallelogram spanned by $1$ and $\tau$.
Hence the quantity $\sup_{\tau\in\H}\lambda_1(L_\lambda)^2/\omega$
is the so-called \emph{Hermite constant} $\gamma_2$ of lattices in $\R^2$.
It is classically known that the Hermite constant is given by $\gamma_2=\frac{2}{\sqrt3}$ (see for instance \cite{Cas}).
This concludes the proof.
\end{proof}

%%%%% K3 surfaces %%%%%
\section{Systolic inequalities for K3 surfaces}
\label{sec:K3}

This section is devoted to prove the following main results of the present article.

\begin{Thm}
\label{mainThm1}
Let $X$ be a complex projective K3 surface. Then
\[
\sys(\sigma)^2\leq C\cdot\vol(\sigma)
\]
for any $\sigma\in\Stab^\dagger(\D^b\Coh(X))$, where
\[
C = \frac{((\rho+2)!)^2|\disc\ \NS(X)|}{2^\rho}+4.
\]
Here $\rho$ and $\disc$ denote the rank and the discriminant of the N\'eron--Severi group $\NS(X)$, respectively.
\end{Thm}

\begin{Thm}
\label{mainThm2}
Let $X$ be a K3 surface of Picard rank one, with $\NS(X)=\Z H$ and $H^2=2n$.
Then
$$
\sys(\sigma)^2\leq 4(n+1)\cdot\vol(\sigma)
$$
holds for any $\sigma\in\Stab^\dagger(\D^b\Coh(X))$.
\end{Thm}

\subsection{Reduction to lattice-theoretic problems}

We start with recalling some standard notations.
Let $X$ be a smooth complex projective K3 surface
and $\D=\D^b\Coh(X)$ be its derived category.
Sending an object $E\in\D$ to its Mukai vector $v(E) = \ch(E)\sqrt{\td(X)}$
identifies the numerical Grothendieck group of $\D$ with the lattice
$$
\NN(\D) \cong H^0(X,\Z) \oplus \NS(X) \oplus H^4(X,\Z).
$$
The Mukai pairing on $\NN(\D)$ is given by
$$
((r_1, D_1, s_1), (r_2, D_2, s_2)) =  D_1\cdot D_2 - r_1s_2 - r_2s_1.
$$
Since the Mukai pairing on $\NN(\D)$ is non-degenerate, any group homomorphism $\ZZ: \NN(D) \ra \C$ can be written as $\ZZ(v) = (\Omega, v)$ for a unique $\Omega \in \NN(\D)\otimes\C$.
This defines a map
\[
\pi\colon\Stab(\D)\ra\NN(\D)\otimes\C
\]
which sends a stability condition $\sigma=(\ZZ_\sigma,\PP_\sigma)$ to the unique element in $\NN(\D)\otimes\C$ associated to its central charge $\ZZ_\sigma:\NN(\D)\ra\C$.

Define
\[
\Delta^+(\D) \coloneqq \{ \delta=(r,D,s)\in\NN(\D)\colon \delta^2=-2 \text{\ and\ } r>0\},
\]
\[
\mathcal{K}(\D) \coloneqq \{ \Omega=\exp(\beta+i\omega)\in\NN(\D)\otimes\C\colon \beta,\omega\in\NS(X)\otimes\R,\ \omega^2>0\},
\]
\[
\mathcal{L}(\D) \coloneqq \{ \Omega \in\mathcal{K}(\D)\colon (\Omega, \delta)\notin\R_{\leq0} \text{\ for all\ } \delta\in\Delta^+(\D)\}.
\]

Now we recall the description of the distinguished connected component $\Stab^\dagger(\D)\subset\Stab(\D)$ by Bridgeland.

\begin{Thm}[\cite{BriK3}]
\label{thm:BridgelandK3}
Let $X$ be a complex projective K3 surface and $\D=\D^b\Coh(X)$.
	\begin{enumerate}[label=(\alph*)]
	\item Let $V(\D)\subset\Stab(\D)$ be the subset consisting of geometric stability conditions constructed via
	tilting of $\Coh(X)$ in \cite[Section 6]{BriK3}. Then the map $\pi$ restricts to give a homeomorphism
	\[
	\pi|_{V(\D)}\colon V(\D)\ra\mathcal{L}(\D).
	\]
	
	\item Let $U(\D)\subset\Stab(\D)$ be the open subset consisting of all geometric stability conditions.
	For any $\sigma\in U(\D)$, there exists a unique $g\in\widetilde{\GL^+(2,\R)}$ such that $\sigma\cdot g\in V(\D)$.
	
	\item Let $\Stab^\dagger(\D)\subset\Stab(\D)$ be the connected component containing $U(\D)$.
	Every stability condition in $\Stab^\dagger(\D)$ is mapped into the closure of $U(\D)$ by some autoequivalence of $\D$.
	\end{enumerate}
\end{Thm}

The following proposition allows us to compute the categorical systole of a stability condition on $\D$ using \emph{only} its central charge.

\begin{Prop}\label{Prop:key}
Let $\sigma=(\ZZ_\sigma, \PP_\sigma)\in\Stab^\dagger(\D)$.
Then
$$
\sys(\sigma) = \min \{ | \ZZ_\sigma(v) | \colon 0 \neq v \in \NN(\D),\ v^2=(v,v) \geq -2 \}.
$$
\end{Prop}

\begin{proof}
%It suffices to prove the statement for $\sigma\in\Stab^\dagger(\D)$,
%since autoequivalences induce Hodge isometries on the numerical Grothendieck group $\NN(\D)$,
%and $\Stab^\Delta(\D)$ is the image of the connected component $\Stab^\dagger(\D)$
%under the $\Aut(\D)$-actions.
Let $v=mv_0\in\NN(\D)$ be a Mukai vector, where $m\in\Z_{>0}$ and $v_0$ is primitive.
A result of Bayer and Macr\`i \cite[Theorem 6.8]{BM},
which is based on a previous result of Toda \cite{TodaK3},
says that if $v_0^2\geq-2$, then there exists a $\sigma$-semistable object with Mukai vector $v$ for \emph{any} $\sigma\in\Stab^\dagger(\D)$.
Hence
\begin{align}
\sys(\sigma) & = \min \{ | \ZZ_\sigma(E) | \colon E \text{ is a } \sigma\text{-semistable object in }\D \}
\notag \\
 &\leq  \min \{ | \ZZ_\sigma(v) | \colon 0 \neq v \in \NN(\D),\ v^2 \geq -2 \}.
\notag
\end{align}

On the other hand, for any stable object $E$,
\begin{align}
v(E)^2 & = -\chi(E, E) = - \hom^0(E,E) + \hom^1(E,E) - \hom^2(E,E)
\notag \\
 & = -2 + \hom^1(E,E) \geq -2.
\notag
\end{align}

Hence
\begin{align}
\sys(\sigma) & = \min \{ | \ZZ_\sigma(E) | \colon E \text{ is a } \sigma\text{-stable object in }\D \}
\notag \\
 &\geq  \min \{ | \ZZ_\sigma(v) | \colon 0 \neq v \in \NN(\D),\ v^2 \geq -2 \}.
\notag
\end{align}

This concludes the proof.
\end{proof}

%\begin{Rmk} \label{Rmk:Mis}
%By Proposition \ref{Prop:key},
%the set of Mukai vectors needed for computing the categorical systole
%is independent of stability condition $\sigma\in\Stab^\Delta(\D)$.
%This is also the case for elliptic curves,
%where the categorical systole is the minimum among the absolute values of central charges 
%of all the nonzero vectors $(d,r)\in\Z^2$.
%However, this is not true in general, see for instance Example \ref{eg:A2}.
%\end{Rmk}

We can now reduce the categorical systolic inequality to a lattice-theoretic problem. To prove Theorem \ref{mainThm1} and \ref{mainThm2}, one needs to find an upper bound of the systolic ratio
\[
\frac{\sys(\sigma)^2}{\vol(\sigma)}
\]
for all $\sigma\in\Stab^\dagger(\D)$.
Since the systolic ratio is a continuous function and is invariant under the actions of autoequivalences, by Theorem \ref{thm:BridgelandK3} (3), it suffices to find an upper bound of the systolic ratio on $U(\D)$.

By Theorem \ref{thm:BridgelandK3} (2), any element in $U(\D)$ can be written as $\sigma\cdot g$ for some $\sigma\in V(\D)$ and $g\in\widetilde{\GL^+(2,\R)}$.
By Lemma \ref{Lem:sys} (4) and the computations in Example \ref{eg:K3}, we have
\[
\frac{\sys(\sigma\cdot g)^2}{\vol(\sigma\cdot g)} \leq
\frac{(1+|t_2|)^2}{1+|t_2|^2} \cdot \frac{\sys(\sigma)^2}{\vol(\sigma)} <
4 \cdot \frac{\sys(\sigma)^2}{\vol(\sigma)}
\]
since $|t_2|<1$ (recall the notations in Lemma \ref{Lem:sys}).
Therefore, it is enough to find an upper bound of the systolic ratio on $V(\D)$.

By Theorem \ref{thm:BridgelandK3} (1), the central charges of stability conditions in $V(\D)$ are of the form
\[
\ZZ_\sigma(v)=(\exp(\beta+i\omega), v)
\]
where $\beta,\omega\in\NS(X)\otimes\R$ and $\omega^2>0$.
By Example \ref{eg:K3}, the volume of the stability conditions of this form is $2\omega^2$.
On the other hand, the categorical systole can be computed by Proposition \ref{Prop:key} given the central charge:
\begin{align*}
\sys(\sigma) & = \min \{ | \ZZ_\sigma(v) | \colon 0 \neq v \in \NN(\D),\ v^2=(v,v) \geq -2 \} \\
& = \min_{\substack{{(r,D,s)\neq(0,0,0)}\\{D^2-2rs\geq-2}}}
\Big\{ \Big|(s+\beta.D+\frac{1}{2}(\beta^2-\omega^2)r)+i(\omega.(D+r\beta))\Big| \Big\}
\end{align*}

Hence, Theorem \ref{mainThm1} and \ref{mainThm2} can be obtained by proving the following lattice-theoretic statements.

\begin{Prop}
\label{prop:latticeThm1}
For any $\beta,\omega\in\NS(X)\otimes\R$ with $\omega^2>0$ and any $C>\frac{(\rho+2)!\sqrt{|\disc\NS(X)|}}{2^{(\rho+1)/2}}$,
there exists $(r,D,s)\in\NN(\D)\bs\{(0,0,0)\}$ satisfying
\begin{enumerate}[label=(\alph*)]
    \item $|\omega.(D+r\beta)|\leq C\sqrt{\omega^2}$;
    \item $|s+\beta.D+\frac{1}{2}(\beta^2-\omega^2)r|\leq \sqrt{2\omega^2}$;
    \item $D^2-2rs\geq-2$.
\end{enumerate}
\end{Prop}

\begin{Prop}
\label{prop:latticeThm2}
For any $\beta\in\R$ and $\omega>0$,
there exists $(r,d,s)\in\Z^3\bs\{(0,0,0)\}$ not all zero satisfying
\begin{enumerate}[label=(\alph*)]
\item 
\[
\frac{|s + 2n(\beta+i\omega)d + n(\beta+i\omega)^2 r|^2}{4n\omega^2} 
<n+1;
\]
\item $nd^2-rs\geq-1$.
\end{enumerate}
\end{Prop}

We will prove these two propositions in the next two subsections.

%%%%%%%%%%%%%%%%%%%%%%%%%%%%%%
\subsection{Systolic inequality for K3 surfaces of Picard rank one}

We prove Proposition \ref{prop:latticeThm2} in this subsection. Note that the method in this proof does not work for K3 surfaces with higher Picard rank, due to the indefiniteness of the intersection pairing on $\NS(X)$ for $\rho(X)>1$.

In order to find such triple $(r,d,s)$ in Proposition \ref{prop:latticeThm2} for small $\omega$, we need the following technical lemma.

\begin{Lem}\label{Lem:est}
For any real number $\beta$ and any $0<\omega<\frac{1}{\sqrt n}$,
there exists integers $(r,d,s)$ such that:
$$
1\leq r\leq\frac{1}{\sqrt n\omega},
$$
$$
|s+2n\beta d+ n\beta^2 r| < \sqrt n\omega,
$$
$$
0\leq nd^2-sr\leq n.
$$
\end{Lem}

\begin{proof}
Let $l=\lfloor\frac{1}{\sqrt n\omega}\rfloor+1$.
For each $1\leq j\leq l$, choose $d_j\in\Z$ such that
$$
-\frac{1}{2} < d_j+ j\beta \leq \frac{1}{2}.
$$

Consider the real numbers $\{2n\beta d_j + n\beta^2 j\}_{1\leq j\leq l}$ modulo 1.
There is at least a pair
$(2n\beta d_j + n\beta^2 j, 2n\beta d_k + n\beta^2 k)$
has distance less than or equals to $1/l$ modulo 1.
Say $j>k$ without loss of generality.
We choose $r=j-k$, $d=d_j-d_k$, and choose $s$ to be the integer closest to
$-2n\beta d - n \beta^2 r$.
Then
$$
1\leq r\leq \lfloor\frac{1}{\sqrt n\omega}\rfloor
$$
and
$$
|s+2n\beta d+ n\beta^2 r| \leq\frac{1}{\lfloor\frac{1}{\sqrt n\omega}\rfloor+1}
< \sqrt n\omega.
$$

Let $\epsilon=s+2n\beta d+ n\beta^2 r$. Then
\begin{align}
nd^2-sr & = nd^2- (- 2n\beta d- n\beta^2 r+\epsilon)r  \notag \\
 & =  n(d+r\beta)^2 - r\epsilon. \notag 
\end{align}

We have
$$
(d+r\beta)^2 = ((d_j+j\beta)-(d_k+k\beta))^2 < 1
$$
and
$$
|r\epsilon| < \frac{1}{\sqrt n\omega} \cdot \sqrt n\omega = 1.
$$

Hence $-1 < nd^2-sr < n+1$. Since it is an integer, thus $0\leq nd^2-sr\leq n$.
\end{proof}

We can now prove Proposition \ref{prop:latticeThm2}.

\begin{proof}[Proof of Proposition \ref{prop:latticeThm2}]
If $\omega\geq\frac{1}{\sqrt n}$, one can simply take the class of skyscraper sheaves $(r,d,s)=(0,0,1)$ and check that it satisfies the required conditions.
If $\omega<\frac{1}{\sqrt n}$, we choose $(r,d,s)$ as in Lemma \ref{Lem:est}. Then it satisfies $sr < nd^2+1$ and
\begin{align}
\frac{|s + 2n(\beta+i\omega)d + n(\beta+i\omega)^2 r|^2}{4n\omega^2} 
& =
\frac{1}{4n}\Big(\frac{s+2n\beta d+n\beta^2 r}{\omega} + n\omega r\Big)^2 + (nd^2-rs)
\notag \\
 &
<\frac{1}{4n}(\sqrt n+\sqrt n)^2+n=n+1.
\notag
\end{align}
\end{proof}

\begin{Rmk}[Spherical systole]\label{Rmk:Sph}
The notion of \emph{spherical objects} in a triangulated category was introduced by Seidel and Thomas \cite{ST}.
These objects are the categorical analogue of Lagrangian spheres in derived Fukaya categories.
%An object $S\in\D^b\Coh(X)$ in the derived category of coherent sheaves on a Calabi--Yau $n$-fold $X$ is called spherical if
%$\Hom^*_\D(S, S)=\C\oplus\C[-n]$.
One can define the \emph{spherical systole} of a Bridgeland stability condition $\sigma$
on a triangulated category $\D$ as
the minimum among the masses of spherical objects:
$$
\sys_\sph(\sigma) \coloneqq \min \{ m_\sigma(S): S \text{ is a spherical object in } \D \}.
$$
\end{Rmk}

It is not hard to show that the notions of spherical systole and categorical systole coincide for stability conditions on elliptic curves. However, this is not true for the derived categories of K3 surfaces. The following proposition shows that the systolic inequality does not hold for spherical systole on K3 surfaces.

\begin{Prop}\label{Prop:Sph}
Let $X$ be a K3 surface of Picard rank one and $\D=\D^b\Coh(X)$ be its derived category of coherent sheaves.
Then
$$
\sup_{\sigma\in\Stab^\dagger(\D)}
\frac{\sys_\sph(\sigma)^2}{\vol(\sigma)} = + \infty.
$$
\end{Prop}

\begin{proof}
Let $\NS(X)=\Z H$ and $H^2=2n$.
Let $\omega H\in\NS(X)\otimes\R$ be an ample class.
Then $\ZZ=(\exp(i\omega H), -)$ gives the central charge of a stability condition in $\Stab^\dagger(\D)$ if $\omega>1$ (\cite[Lemma 6.2]{BriK3}).
By Proposition \ref{Prop:key} and Example \ref{eg:K3},
$$
\sup_{\sigma\in\Stab^\dagger(\D)}
\frac{\sys_\sph(\sigma)^2}{\vol(\sigma)} 
\geq
\sup_{\omega > 1}
\frac{\min\{ |(\exp(i\omega H), v)|^2 : v^2=-2    \}}
{4n\omega^2}.
$$
Here we use the fact that
the Mukai vector of a spherical object $S$ satisfies
$
v(S)^2 = - \hom^0(S,S) + \hom^1(S,S) - \hom^2(S,S) = -2.
$

Let $v=(r, dH, s) \in \Z \oplus \NS(X) \oplus \Z$ be a vector
satisfying $v^2 = 2nd^2 - 2rs = -2$.
Then $rs\neq0$ and 
$r,s$ are both positive or both negative.
Hence
\begin{align}
\sup_{\sigma\in\Stab^\dagger(\D)}
\frac{\sys_\sph(\sigma)^2}{\vol(\sigma)} 
&
\geq
\sup_{\omega > 1}
\frac{\min\{ |(\exp(i\omega H), v)|^2 : v^2=-2    \}}
{4n\omega^2}
\notag \\
 &
=
\sup_{\omega>1}
\min_{nd^2-rs=-1}
\frac{1}{4n} \Big( \frac{s}{\omega} + n\omega r \Big)^2 -1
\notag \\
 &
\geq
\sup_{\omega>1}
\frac{1}{4n} \Big( \frac{1}{\omega} + n\omega  \Big)^2 -1
= + \infty.
\notag
\end{align}

\end{proof}

\begin{Rmk}
In the proof of Proposition \ref{prop:latticeThm2},
we do not make use of the Mukai vectors of spherical objects.
Recall that
$$
\sys(\sigma) = \min \{ | \ZZ_\sigma(v) | : v^2 \geq -2, v\neq 0 \},
$$
but we only use the Mukai vectors that satisfy $0\leq v^2 \leq 2n$ to prove the systolic inequality (c.f.~Lemma \ref{Lem:est}).
\end{Rmk}

%%%%%%%%%%%%%%%%%%%%%%%%%%%%%%
\subsection{Systolic inequality for general K3 surfaces}

We prove Proposition \ref{prop:latticeThm1} in this subsection. The proof is based on a classical result on the existence of integer points by Minkowski.

\begin{Thm}[Minkowski]
Every convex set in $\R^n$ which is symmetric with respect to the origin and has volume greater than $2^n$ contains a non-zero integer point.
\end{Thm}

\begin{proof}[Proof of Proposition \ref{prop:latticeThm1}]
We fix an identification $\NS(X)\cong\Z^\rho$.
%and let $Q$ be the quadratic form on $\R^\rho$ given by the intersection pairing on the middle cohomology of $X$.
For any $\omega\in\NS(X)\otimes\R$ with $\omega^2>0$, the intersection pairing restricts on $\langle\omega\rangle^\perp\subset\R^\rho$ is negative definite by Hodge index theorem. Choose
\[
D_1,\cdots,D_{\rho-1}\in\langle\omega\rangle^\perp\subset\R^\rho
\]
such that $D_i.D_j=0$ for all $i\neq j$, and $D_i^2=-1$ for all $i$.

Consider the following vectors in $\R^{\rho+2}\cong\R\oplus(\NS(X)\otimes\R)\oplus\R$:
\[
\begin{array}{rl}
v_1 & = \sqrt2\cdot (0, D_1, -\beta.D_1), \\
v_2 & =  \sqrt2\cdot  (0, D_2, -\beta.D_2), \\
... \\
v_{\rho-1} & =  \sqrt2\cdot  (0, D_{\rho-1}, -\beta.D_{\rho-1}),  \\
v_{\rho} & =  \sqrt2\cdot ( \frac{1}{\sqrt{\omega^2}}, \frac{-\beta}{\sqrt{\omega^2}},  \frac{\beta^2+\omega^2}{2\sqrt{\omega^2}}), \\
v_{\rho+1} & =  C\cdot (0, \frac{\omega}{\sqrt{\omega^2}},  \frac{-\beta.\omega}{\sqrt{\omega^2}}  ),\\
v_{\rho+2} & =  \sqrt2\cdot  (0,  0,  \sqrt{\omega^2}),
\end{array}
\]
where $C$ is any number larger than $\frac{(\rho+2)!\sqrt{|\disc\NS(X)|}}{2^{(\rho+1)/2}}$.

Consider the following convex set in $\R^{\rho+2}$ which is symmetric with respect to the origin:
\[
    \cC:=\{ \sum_{i=1}^{\rho+2} e_iv_i: \sum_{i=1}^{\rho+2}|e_i|\leq1\}.
\]
One can check that the volume of $\cC$ is greater than $2^{\rho+2}$.
Therefore to prove Proposition \ref{prop:latticeThm1}, it suffices to show that any vector in $\cC$ satisfies the following conditions:
\begin{enumerate}[label=(\alph*)]
    \item $|\omega.(D+r\beta)|\leq C\sqrt{\omega^2}$;
    \item $|s+\beta.D+\frac{1}{2}(\beta^2-\omega^2)r|\leq \sqrt{2\omega^2}$;
    \item $D^2-2rs\geq-2$.
\end{enumerate}

Observe that the vectors $v_1,v_2,\ldots,v_{\rho+2}$ satisfy all three conditions. Since the first two conditions are linear in $(r,D,s)\in\R^{\rho+2}$, hence are satisfied by all the vectors in $\cC$.

Let $(r,D,s)=\sum_{i=1}^{\rho+2} e_iv_i\in\cC$. Then
\begin{align*}
D^2-2rs & =e_{\rho+1}^2C^2 - 2(e_1^2+\cdots+e_\rho^2)-4e_\rho e_{\rho+2} \\
& \geq -2(e_1^2+\cdots+e_{\rho-1}^2)-2(e_\rho+e_{\rho+2})^2 \\
& \geq -2
\end{align*}
since $\sum_{i=1}^{\rho+2}|e_i|\leq1$.
This concludes the proof of Proposition \ref{prop:latticeThm1} by Minkowski's theorem.
\end{proof}

%Consider
%\begin{multline*}
%(e_1A_1D_1+\cdots+e_{\rho-1}A_{\rho-1}D_{\rho-1}+e_{\rho}A_\rho\frac{-B}{\sqrt{\omega^2}} + e_{\rho+1}C_1 \frac{\omega}{\sqrt{\omega^2}})^2\\ 
%- 2(e_\rho A_\rho\frac{1}{\sqrt{\omega^2}})
%(-e_1A_1B.D_1-\cdots-e_{\rho-1}A_{\rho-1}B.D_{\rho-1}+e_\rho A_\rho\frac{B^2+\omega^2}{2\sqrt{\omega^2}}
%+e_{\rho+1}C_1\frac{-B.\omega}{\sqrt{\omega^2}}
%+e_{\rho+2}C_2\sqrt{\omega^2})
%\end{multline*}
%\[
%=e_{\rho+1}^2C_1^2 - (e_1^2A_1^2+\cdots+e_\rho^2A_\rho^2)-2e_\rho e_{\rho+2} A_\rho C_2
%\]

%For instance, take $A_1=\cdots=A_{\rho-1}=\sqrt2$, $A_\rho=1$, $C_2=2$. Then
%\[
%\begin{array}{ll}
%& e_1^2A_1^2+\cdots+e_\rho^2A_\rho^2 + 2e_\rho e_{\rho+2} A_\rho C_2 \\
%=&2(e_1^2+\cdots+e_{\rho-1}^2)+e_\rho^2+4e_\rho e_{\rho+2} \\
%\leq&2(|e_1|+\cdots+|e_{\rho-1}|)+|e_\rho|+|e_{\rho}+e_{\rho+2}| <2.
%\end{array}
%\]

%%%%% Future %%%%%
\section{Future studies}
\label{sec:future}

\subsection*{Applications in symplectic geometry}

Recall the two questions we proposed in the introduction.

\begin{Question}[see Question \ref{Q:A}]
\label{Q:A-5}
Let $Y$ be a Calabi--Yau manifold and $\omega$ be a symplectic form on $Y$.
Does there exist a constant $C=C(Y,\omega)>0$ such that
$$
\sys(Y,\omega,\Omega)^2
 \leq C\cdot  \Big|\int_Y\Omega\wedge\overline\Omega\Big|
$$
holds for any holomorphic top form $\Omega$ on $Y$?
\end{Question}

\begin{Question}[see Question \ref{Q:B}]
\label{Q:B-5}
Let $X$ be a Calabi--Yau manifold and $\Omega$ be a complex structure on $X$.
Let $\D=\D^b\Coh(X, \Omega)$ be its derived category of coherent sheaves.
Does there exist a constant $C=C(X, \Omega)>0$ such that
$$
\sys(\sigma)^2 \leq C \cdot
\vol(\sigma)
$$
holds for any $\sigma\in\Stab^\dagger(\D)$?
%Here $\Stab^\dagger(\D)$ denotes the distinguished connected component of $\Stab(\D)$ that contains geometric stability conditions.
\end{Question}

In the present article, we give an affirmative answer to Question \ref{Q:B-5} for any complex projective K3 surface $X$, and find an explicit systolic constant $C$ which depends only on the rank and discriminant of $\NS(X)$. Now we discuss how this result can be used to answer Question \ref{Q:A-5}, via the homological mirror symmetry conjecture proposed by Kontsevich.

\begin{Conj}[\cite{Kon}]
\label{conj:HMS}
For any Calabi--Yau manifold with a symplectic structure $(Y,\omega)$ , there exists a Calabi--Yau manifold with a complex structure $(X,\Omega)$ such that there is an equivalence between triangulated categories
\[
\D^b\Fuk(Y,\omega)\cong\D^b\Coh(X,\Omega).
\]
\end{Conj}

The conjecture has been proved in several cases, see for instance \cite{PZHMS, SeidelK3, SheridanCY}.
In particular, for any K3 surface $Y$ and a symplectic form $\omega$ on $Y$, it is expected that there exists a K3 surface $X$ with a complex structure $\Omega$ such that the above equivalence of categories holds.

\begin{Cor}
Assume Conjecture \ref{conj:BridgelandJoyce} and \ref{conj:HMS} hold for all K3 surfaces. Also assume that the space of stability conditions on $\D^b\Coh(X)$ is connected for any projective K3 surface $X$.
Then Question \ref{Q:A-5} has an affirmative answer for any K3 surface $Y$.
\end{Cor}

\begin{proof}
Conjecture \ref{conj:BridgelandJoyce} holds for $\D^b\Fuk(Y,\omega)$ implies that for any holomorhpic $2$-form $\Omega_Y$ on $Y$, there is an associated Bridgeland stability condition $\sigma_{\Omega_Y}$ on the derived Fukaya category of $(Y,\omega)$ such that
\[
\sys(Y,\omega,\Omega_Y)=\sys(\sigma_{\Omega_Y}) \text{\ \ and\ \ } \Big|\int_Y\Omega_Y\wedge\overline\Omega_Y\Big|=\vol(\sigma_{\Omega_Y}).
\]

Assuming the validity of Conjecture \ref{conj:HMS} for K3 surface $(Y,\omega)$, then there exists a K3 surface with a complex structure $(X,\Omega)$ such that
\[
\D^b\Fuk(Y,\omega)\cong\D^b\Coh(X,\Omega).
\]
Therefore $\sigma_{\Omega_Y}$ induces a Bridgeland stability condition on $\D^b\Coh(X,\Omega)$.
By Theorem \ref{thm:K3} and the assumption that $\Stab(\D^b\Coh(X))$ is connected, we have
\[
\sys(\sigma_{\Omega_Y})^2\leq C\cdot\vol(\sigma_{\Omega_Y}),
\]
where $C = \frac{((\rho+2)!)^2|\disc\ \NS(X)|}{2^\rho}+4$. Hence, the inequality 
$$
\sys(Y,\omega,\Omega_Y)^2
 \leq C\cdot  \Big|\int_Y\Omega_Y\wedge\overline\Omega_Y\Big|
$$
holds for any holomorphic $2$-form $\Omega_Y$ on $Y$.
\end{proof}

\subsection*{Systolic inequality for Calabi--Yau threefolds}
The existence of Bridgeland stability conditions on quintic Calabi--Yau threefolds is proved by Li recently \cite{LiQuintic}. It would be interesting to investigate whether the categorical systolic inequality
\[
\sys(\sigma)^2\leq C\cdot\vol(\sigma)
\]
continues to hold for Calabi--Yau threefolds. Note that the categorical volumes of geometric stability conditions on quintic Calabi--Yau threefolds near the ``large volume limit" were computed in \cite[Section 4.4]{FKY}. Hence the main difficulty lies in determining the categorical systoles. It would be nice if properties similar to Proposition \ref{Prop:key} hold for Calabi--Yau threefolds.

\subsection*{Categorical systole as topological Morse function}
It is proved by Akrout \cite{Akr} that the systole of Riemann surfaces is a \emph{topological Morse function} on the Teichm\"uller space.
In fact, Akrout shows that a ``generalized systolic function" defined locally as the minimum of a finite number of functions with positive definite Hessians is a topological Morse function,
and then apply a result of Wolpert \cite{Wol} that length functions have positive definite Hessians with respect to the Weil--Petersson metric on the Teichm\"uller space.
 
Motivated by the correspondence between flat surfaces and stability conditions,
it would be interesting to show that the categorical systole is a topological Morse function on the space of Bridgeland stability conditions,
and then deduce some topological properties of the quotient space $\Stab(\D)/\Aut(\D)$.
Since there also is a categorical analogue of Weil--Petersson metric on the space of stability conditions studied in \cite{FKY},
one might be able to follow the same approach as Akrout's proof.
We hope to come back to this question in the future.

%\subsection*{Optimal systolic ratio of K3 surfaces of Picard rank one}
%It is not clear whether the constant $n+1$ in the categorical systolic inequality
%$$
%\sys(\sigma)^2 \leq (n+1) \vol(\sigma)
%$$
%is optimal.
%As discussed in the proof of Theorem 4.1,
%finding the optimal systolic ratio is equivalent to solving the following lattice-theoretic problem:
%$$
%\sup_{\substack{\beta\in\R \\ \omega>0}}
%\min_{\substack{(s,d,r)\in\Z^3\\ sr\leq nd^2+1  \\(s,d,r)\neq(0,0,0)}}
%\frac
%{|s + 2n(\beta+i\omega)d + n(\beta+i\omega)^2 r|^2}
%{4n\omega^2}
%=?
%$$
%
%One might expect the optimal is achieved by
%certain stability conditions on K3 surfaces with ``extra symmetries",
%like the optimal of Loewner's torus systolic inequality is
%achieved by the flat equilateral torus.

%
%\subsection*{Miscellaneous}
%
%It is mentioned in Remark \ref{Rmk:Mis} that in the case of elliptic curves and K3 surfaces,
%the set of Mukai vectors needed for computing the categorical systole is independent of stability conditions.
%This is also true for the derived categories of coherent sheaves on smooth projective curves of genus $\geq1$,
%since there is no wall-crossings on the space of stability conditions \cite{MacriCurve}.
%It would be interesting to find other triangulated categories that also satisfy this property.

%\item Akrout \cite{Akr} proves that the systole of Riemann surfaces is a \emph{topological Morse function} on the Teichm\"uller space.
%It would be interesting to know whether the categorical systole is also a topological Morse function on the space of Bridgeland stability conditions.
%\end{itemize}

\bibliographystyle{alpha}
%\bibliography{biblio} 

\begin{thebibliography}{9999999}

\bibitem[Akr03]{Akr}
H.~Akrout.
Singularit\'{e}s topologiques des systoles g\'{e}n\'{e}ralis\'{e}es.
\emph{Topology},
42(2):291--308, 2003.

\bibitem[Ati57]{Atiyah}
M.~F.~Atiyah.
Vector bundles over an elliptic curve.
\emph{Proc. London Math. Soc. (3)},
7:414--452, 1957.

\bibitem[BM14]{BM}
A.~Bayer and E.~Macr\`\i.
Projectivity and birational geometry of {B}ridgeland moduli spaces.
\emph{J. Amer. Math. Soc.},
27(3):707--752, 2014.

\bibitem[Bri07]{Bri}
T.~Bridgeland.
Stability conditions on triangulated categories.
\emph{Ann. of Math. (2)},
166(2):317--345, 2007.

\bibitem[Bri08]{BriK3}
T.~Bridgeland.
Stability conditions on {$K3$} surfaces.
\emph{Duke Math. J.},
141(2):241--291, 2008.

\bibitem[Bri09]{BriSurvey}
T.~Bridgeland.
Spaces of stability conditions.
In \emph{Algebraic geometry---{S}eattle 2005. {P}art 1},
volume 80 of \emph{Proc. Sympos. Pure Math.},
pages 1--21.
Amer. Math. Soc., Providence, RI,
2009.

\bibitem[BS15]{BS}
T.~Bridgeland and I.~Smith.
Quadratic differentials as stability conditions.
\emph{Publ. Math. Inst. Hautes \'{E}tudes Sci.},
121:155--278, 2015.

\bibitem[Cas97]{Cas}
J.~W.~S.~Cassels.
\emph{An introduction to the geometry of numbers}.
Classics in Mathematics.
Springer-Verlag, Berlin,
1997.

\bibitem[DRY06]{DRY}
M.~R.~Douglas, R.~Reinbacher, and S.-T.~Yau.
Branes, Bundles and Attractors: Bogomolov and Beyond,
2006.
arXiv:math/0604597.

\bibitem[FKY17]{FKY}
Y.-W.~Fan, A.~Kanazawa, and S.-T.~Yau.
Weil--Petersson geometry on the space of Bridgeland stability conditions,
2017.
arXiv:1708.02161.

\bibitem[Gin06]{Gin}
V.~Ginzburg.
Calabi--Yau algebras,
2006.
arXiv:math/0612139.

\bibitem[GMN13]{GMN}
D.~Gaiotto, G.~W.~Moore, and A.~Neitzke.
Wall-crossing, {H}itchin systems, and the {WKB} approximation.
\emph{Adv. Math.},
234:239--403, 2013.

\bibitem[Hai18]{Hai}
F.~Haiden.
An extension of the Siegel space of complex abelian varieties and conjectures on stability structures,
2018.
arXiv:1808.06364.

\bibitem[HKK17]{HKK}
F.~Haiden, L.~Katzarkov, and M.~Kontsevich.
Flat surfaces and stability structures.
\emph{Publ. Math. Inst. Hautes \'{E}tudes Sci.},
126:247--318, 2017.

\bibitem[Joy15]{Joyce}
D.~Joyce.
Conjectures on {B}ridgeland stability for {F}ukaya categories of {C}alabi-{Y}au manifolds, special {L}agrangians,
and {L}agrangian mean curvature flow.
\emph{EMS Surv. Math. Sci.},
2(1):1--62, 2015.

\bibitem[Kat07]{Katz}
M.~G.~Katz.
\emph{Systolic geometry and topology},
volume 137 of
\emph{Mathematical Surveys and Monographs}.
American Mathematical Society, Providence, RI, 2007.
With an appendix by J.~P.~Solomon.

\bibitem[Kel11]{Kel}
B.~Keller.
Deformed {C}alabi-{Y}au completions.
\emph{J. Reine Angew. Math.},
654:125--180, 2011.
With an appendix by M.~Van~den~Bergh.

\bibitem[Kon95]{Kon}
M.~Kontsevich.
Homological algebra of mirror symmetry.
In \emph{Proceedings of the {I}nternational {C}ongress of {M}athematicians, {V}ol. 1, 2 ({Z}\"{u}rich, 1994)},
page 120--139.
Birkh\"{a}user, Basel, 1995.

\bibitem[KS08]{KS}
M.~Kontsevich and Y.~Soibelman.
Stability structures, motivic Donaldson--Thomas invariants and cluster transformations,
2008.
arXiv:0811.2435.

\bibitem[KS14]{KS2}
M.~Kontsevich and Y.~Soibelman.
Wall-crossing structures in {D}onaldson-{T}homas invariants, integrable systems and mirror symmetry.
In \emph{Homological mirror symmetry and tropical geometry},
volume 15 of \emph{Lect. Notes Unione Mat. Ital.},
pages 197--308.
Springer, Cham, 2014.

\bibitem[Li19]{LiQuintic}
C.~Li.
On stability conditions for the quintic threefold.
\emph{Invent. Math.},
218(1):301--340, 2019.

\bibitem[Mac07]{Macri}
E.~Macr\`i.
Stability conditions on curves.
\emph{Math. Res. Lett.},
14(4):657--672, 2007.

\bibitem[Moo98]{Moo}
G.~W.~Moore.
Arithmetic and Attractors,
1998.
arXiv:hep-th/9807087.

\bibitem[Oka06]{Okada}
S.~Okada.
Stability manifold of $\P^1$.
J. Algebraic Geom., 15(3):487--505, 2006

\bibitem[Pac19]{Pacini}
T.~Pacini.
Extremal length in higher dimensions and complex systolic inequalities,
2019.
arXiv:1904.07807.

\bibitem[Pol03]{Pol}
A.~Polishchuk.
\emph{Abelian varieties, theta functions and the {F}ourier transform},
volume 153 of \emph{Cambridge Tracts in Mathematics}.
Cambridge University Press, Cambridge, 2003.

\bibitem[PZ98]{PZHMS}
A.~Polishchuk and E.~Zaslow.
Categorical mirror symmetry: the elliptic curve.
\emph{Adv. Theor. Math. Phys.},
2(2):443--470, 1998.

\bibitem[Sei15]{SeidelK3}
P.~Seidel.
Homological mirror symmetry for the quartic surface.
\emph{Mem. Amer. Math. Soc.},
236(1116):vi+129, 2015.

\bibitem[She15]{SheridanCY}
N.~Sheridan.
Homological mirror symmetry for {C}alabi-{Y}au hypersurfaces in projective space.
\emph{Invent. Math.},
199(1):1--186, 2015.

\bibitem[ST01]{ST}
P.~Seidel and R.~Thomas.
Braid group actions on derived categories of coherent sheaves.
\emph{Duke Math. J.},
108(1):37--108, 2001.

\bibitem[Tod08]{TodaK3}
Y.~Toda.
Moduli stacks and invariants of semistable objects on {$K3$} surfaces.
\emph{Adv. Math.},
217(6):2736--2781, 2008.

\bibitem[Wol87]{Wol}
S.~A.~Wolpert.
Geodesic length functions and the {N}ielsen problem.
\emph{J. Differential Geom.},
25(2):275--296, 1987.

\end{thebibliography}

\noindent Y.-W.~Fan \\
\textsc{Yau Mathematical Sciences Center, Tsinghua University}\\
\texttt{yuweifanx@gmail.com}

\end{document}